\documentclass[11pt]{amsart}
\usepackage{times}
\usepackage[T1]{fontenc}
\usepackage[utf8]{inputenc}
\usepackage{lmodern}
\usepackage{color}
\usepackage[english]{babel}
\usepackage{geometry}
\usepackage{relsize}
\usepackage{booktabs}
\usepackage{mathtools}
\usepackage{amsthm}
\usepackage{amsmath,amssymb}
\usepackage{emptypage}
\usepackage{xfrac}    
\usepackage{faktor}
\usepackage{graphicx}
\usepackage{listings}
\usepackage{lipsum}
\usepackage{setspace}
\usepackage{lipsum}
\usepackage{newlfont}
\usepackage[mathscr]{eucal}	
\usepackage{physics}
\usepackage{tikz}
\usepackage{url}
\usepackage{hyperref}
\usetikzlibrary{positioning}
\usepackage{xcolor}
\usepackage[new]{old-arrows}
\usepackage{multirow}

\theoremstyle{plain}
\newtheorem{theorem}{Theorem}[section]
\newtheorem{lemma}[theorem]{Lemma}
\newtheorem{proposition}[theorem]{Proposition}
\newtheorem{corollary}[theorem]{Corollary}

\theoremstyle{definition}
\newtheorem{definition}[theorem]{Definition}

\newtheorem{example}[theorem]{Example}

\newtheorem{question}[theorem]{Question}
\newtheorem{remark}{Remark}[section]
\usepackage{setspace}
\usepackage[signatures,write]{frontespizio}

\usepackage{fancyhdr}

\newcommand{\bigslant}[2]{{\raisebox{.2em}{$#1$}\left/\raisebox{-.2em}{$#2$}\right.}}

\newcommand{\Sph}{\mathbb{S}} 
\newcommand{\Z}{\mathbb{Z}} 
\newcommand{\R}{\mathbb{R}} 
\newcommand{\C}{\mathbb{C}} 

\newcommand\restrict[1]{\raisebox{-.5ex}{$|$}_{#1}}
\AtBeginDocument
{
\setlength\abovedisplayskip{.6\baselineskip}
\setlength\abovedisplayshortskip{.6\baselineskip}
\setlength\belowdisplayskip{.5\baselineskip}
\setlength\belowdisplayshortskip{.5\baselineskip}
}
\begin{document}

\title{CYT and SKT manifolds with parallel  Bismut torsion} 

\author{Beatrice Brienza}
\address[Beatrice Brienza]{Dipartimento di Matematica ``G. Peano'', Universit\`{a} degli studi di Torino \\
Via Carlo Alberto 10\\
10123 Torino, Italy}
\email{beatrice.brienza@unito.it}

\author{Anna Fino}
\address[Anna Fino]{Dipartimento di Matematica ``G. Peano'', Universit\`{a} degli studi di Torino \\
Via Carlo Alberto 10\\
10123 Torino, Italy\\
\& Department of Mathematics and Statistics, Florida International University\\
Miami, FL 33199, United States}
\email{annamaria.fino@unito.it, afino@fiu.edu}

\author{Gueo Grantcharov}
\address[Gueo Grantcharov]{Department of Mathematics and Statistics \\
Florida International University\\
Miami, FL 33199, United States}
\email{grantchg@fiu.edu}



\maketitle

\begin{abstract}  In the present paper, we study compact complex manifolds admitting a Hermitian metric which is SKT and CYT and whose Bismut torsion is parallel. We first obtain a  characterization of  the universal cover of such manifolds as a product of a K\"ahler Ricci-flat manifold with a Bismut flat one. Then, using  a mapping torus construction,  we provide non-Bismut flat examples. The existence of generalized K\"ahler structures is also investigated.

\end{abstract}

\section{Introduction}
Let $(M,J,g)$ be a Hermitian manifold of complex dimension $n$  with fundamental form $\omega=g(J\cdot,\cdot)$. A connection $\nabla$ on $TM$ is said to be Hermitian if $\nabla g=0$ and $\nabla J=0$. In \cite{PG} an affine line of Hermitian connections is introduced. These are known as \emph{Gauduchon} or \emph{canonical} connections and they can be written as
\begin{equation} \label{eqn:canonical}
g(\nabla^t_XY,Z)=g(\nabla^{LC}_XY,Z)+\frac{t-1}{4}(d^c\omega)(X,Y,Z)+\frac{t+1}{4}(d^c\omega)(X,JY,JZ),
\end{equation}
where $d^c \omega=-Jd\omega$. We adopt the convention $Jd\omega(X,Y,Z):=d\omega(JX,JY,JZ)$.

When $(M,g,J)$ is K\"ahler, $d^c\omega$ is zero, and the line collapses to a single point, which is the Levi-Civita connection. However, when  $(M,g,J)$ is not K\"ahler, the line is not trivial and the connections $\nabla^t$ have non vanishing torsion. 
For particular values of $t \in \R$, the Chern and the Bismut connections are recovered. More precisely, $\nabla^{1}=\nabla^{Ch}$ and $\nabla^{-1}=\nabla^{B}$ (\cite{Bi, CH}). Although $\nabla^{LC}, \nabla^{B}, \nabla^{Ch}$ are mutually different connections, any one of them completely determines the other two.  

The Bismut connection, also known as the Strominger connection \cite{Strominger},  can  be  characterized as the only Hermitian connection with totally skew-symmetric torsion. It follows from \eqref{eqn:canonical} that its expression is given by
\begin{equation*} 
g(\nabla^B_XY,Z)=g(\nabla^{LC}_XY,Z)-\frac{1}{2}d^c\omega(X,Y,Z),
\end{equation*}
and its torsion $3$-form $H$ is
$$
H(X, Y,  Z) = g(T^B (X, Y), Z) = d \omega (JX, JY, JZ) = - d^c \omega (X, Y, Z).
$$
If the torsion $3$-form $H$ is closed, i.e.,  $d d^c \omega =0$ or, equivalently, $\partial \overline \partial \omega=0$, the metric $g$ is said {\em strong K\"ahler with torsion} (SKT in short) or {\em pluriclosed}. \\
The Hermitian metric $g$  is said to be {\em Calabi-Yau with torsion} (CYT in short) if the associated Bismut–Ricci curvature $\rho^B(g)$ vanishes, where $\rho^B(g)$ is given, up to a constant factor, by tracing the Bismut curvature tensor $R^B(g)$ in the endomorphism components, i.e., 
\[
\rho^B (X,Y) = \frac{1}{2} \  \sum_{i=1}^{2n} R^B(X,Y, Je_i, e_i),
\]
where $\{e_i\}$ is an orthonormal frame of the tangent space of $M$ at a given point. This is a natural Ricci-type curvature which coincides with the usual Ricci form when the metric $g$ is K\"ahler.
Since the connection  $\nabla^B$ is  Hermitian, it determines via its Ricci form  a representative of the first Chern class in de Rham cohomology of $M$. As a consequence, the existence of  a CYT structure  is obstructed by having vanishing first Chern class.  \\
A  SKT  structure $(M,J,g)$ which is also CYT is known in literature as \emph{Bismut Hermitian-Einstein} (BHE in short).  Taking inspiration from the Calabi-Yau Theorem (\cite{CAL,CAL2,YAU,YAU2}), Garcia-Fernandez, Jordan and Streets investigated in \cite{GFS} whether the condition $c_1(M) = 0$ guarantees the existence of a Bismut Hermitian-Einstein metric. They actually proved that the answer is far from being positive, showing that in every complex  dimension there exist infinitely many complex manifolds with vanishing first Chern class which do not admit a Bismut Hermitian-Einstein metric. An additional obstruction for the existence of CYT metric on a compact complex manifold with vanishing first Chern class arises from the Gauduchon's plurigenera vanishing theorem. It is noticed in \cite{AI}, and in particular shows that a manifold which has Kodaira dimension at least one, can not admit CYT metrics. Using the construction of $T^2$-bundles over complex surfaces of general type with negative K\"ahler-Einstein metrics, one can provide infinite number of complex 3-dimensional manifolds with vanishing first Chern class which do not admit CYT metric. \\
In \cite{ST}, Streets and Tian introduced the \emph{pluriclosed flow},  an evolution equation for Hermitian metrics preserving the SKT condition and the existence of generalized K\"ahler structures (\cite{ST2}).  If $(M,J)$ is a complex manifold and $\omega_0$ is a SKT metric on $M$, the pluriclosed flow with initial data $\omega_0$ evolves as
\begin{equation*} 
\begin{cases}
\frac{\partial}{\partial t} \omega(t)=-\big(\rho^B (t)\big)^{(1,1)} \\
\omega(0)=\omega_0
\end{cases}
\end{equation*}
and its static points are given by Hermitian structures such that
\begin{equation} \label{eqn:staticpoints}
\big(\rho^B \big)^{(1,1)}=\lambda \omega \ \  \text{for} \  \lambda \in \R.
\end{equation}
When a Hermitian structure is SKT and CYT, then \eqref{eqn:staticpoints} is satisfied with $\lambda=0$, motivating a huge interest in finding explicit non-trivial examples of BHE manifolds, where by non trivial we mean not  diffeomorphic to a product of a K\"ahler-Ricci flat manifold and a Bismut flat space. Moreover, up to now, no non-trivial examples of BHE manifolds are known. Non-Hermitian  homogeneous examples which are Bismut Ricci-flat, with  closed torsion form and non-vanishing Bismut curvature have been constructed in \cite{PR1,PR2}, but no Hermitian  homogeneous examples  of this type are known.  Some negative results on compact semisimple Lie groups and more in general on C-spaces have been given in \cite{Barbaro, FG}. Furthermore, Ivanov and Stanchev proved in \cite{IS} that a compact  SKT and  CYT $6$-manifold is Bismut Ricci flat if, and only if, either the torsion has constant norm or the Riemannian scalar curvature is constant. In particular, the torsion is harmonic. \\
The Bismut flatness condition has been largely investigated in literature (\cite{AF,CS,CS2,WO,WO2}), and Bismut flat compact manifolds and Bismut flat simply connected manifolds were completely characterized by Wang, Yang and Zheng in \cite{WYZ}. We recall that 
a \emph{Samelson space} is a Hermitian manifold $(G'=G \times \R^n, g'= b+g_E,J_L)$, where $G$ is a compact, connected and simply connected semisimple Lie group, $g'=b+g_E$ is the bi-invariant metric on $G'$ given by the product of the bi-invariant metric $b$ on $G$ and the euclidean metric $g_E$, and $J_L$ is a left invariant complex structure compatible with $g'$. 
A group homomorphism $\rho:\Z^n \to {\mbox {Isom}}(G)$ induces a free and properly discontinuous action of $\Z^n$ on $G'$ as isometries via
\[
m \cdot (p,t) \mapsto (\rho(m)p, t+m).
\]
If $J_L$ is preserved by  this action of $\Z^n$, then the compact quotient $G' / \Z^n$ inherits the structure of a complex manifold by $G'$. In this case, $G' / \Z^n$ is said to be a \emph{local Samelson space}.

By  \cite{WYZ}, if  $(M,g,J)$  is  a compact Bismut flat Hermitian manifold, then there exists a finite unbranched cover $M'$ of $M$ which is a local Samelson space. Moreover, if  $(M,g,J)$ is  Bismut-flat  simply-connected Hermitian manifold, then there exists a Samelson space $(G',g',J')$ such that $M$ is an open complex submanifold of $G'$ and $g = g' \restrict{M}$.

\smallskip
In this paper we will mainly focus on the case when the torsion of Bismut connection is parallel with respect to the Bismut connection.  According to \cite{ZZ2}, we will call such a manifold (and sometimes also the Hermitian metric) Bismut torsion parallel (BTP in short). In \cite{ZZ3}, it has been proved that a Hermitian metric is BTP and SKT if, and only if, it is Bismut K\"ahler-like, namely, if the Bismut curvature satisfied the first Bianchi identity and the type condition (see Section \ref{section:one} for further details). 
The BTP condition was studied by several authors, e.g., \cite{AV, PZ, ZZ2, ZZ3} and in the references therein. In the non-hermitian setting, it was proved in \cite{AFF} that a complete and simply connected Riemannian manifold admitting a connection with parallel and closed skew-torsion is, up to products, a Lie group. \\
Our first main result characterizes the universal cover of SKT and CYT manifolds with parallel Bismut torsion. We prove the following 
\begin{theorem} 
Let $(M, I)$ be a compact complex manifold admitting a SKT and CYT  $I$-Hermitian metric  $h$ such that its  Bismut torsion $3$-form $H$  is parallel, i.e. $\nabla^B H=0$. Then, the Riemannian holomorphic universal cover $(\tilde M,\tilde I,  \tilde h)$  of $(M, I, h)$  is holomorphically isometric to the product $(M_1,  J_1, g_1) \times (M_2,  J_2, g_2)$, where $(M_1, J_1, g_1)$ is a K\"ahler Ricci flat manifold and $(M_2, J_2,  g_2)$ is a Samelson space.
\end{theorem}
In Section \ref{section:two} we use the above characterization to construct non-trivial examples of CYT and SKT Hermitian manifolds (Propositions \ref{prop:flathopf} and \ref{prop:flatsam}). \\
With a similar technique used by the first two authors in \cite{BF}, in Lemma \ref{lemma:pluriclosed} we show the existence of a SKT structure $(I, h)$ on mapping tori of type $M_f=(K \times \Sph^3)_f$, with $f=(\psi,Id_{\Sph^3})$, where $K$ is a compact K\"ahler manifold and $\psi$ is a K\"ahler isometry of $K$. We point out that such mapping tori do not admit any  K\"ahler  metric as their first Betti numbers are odd (Proposition \ref{proposition:prop}).  If in addition $K$ is K\"ahler Ricci-flat  and $\psi$ preserves the K\"ahler Ricci-flat metric,  the SKT structure $(I, h)$ is CYT.
Furthermore, we observe that  the above Hermitian  structure $(I, h)$ is BTP (Remark \ref{remark:remark}) and, whenever $K$ is a K3 surface and  $\psi \neq Id_K$, the mapping tori $M_f$ are non-trivial, i.e., they are not diffeomorphic to a global product of $K$ with the Hopf Surface (Corollary \ref{corollary:nonproduct}).  Hence, we provide the first known examples of non-trivial Bismut Hermitian-Einstein manifolds.\\
It is natural to ask the following
\begin{question}
Does there exist a Bismut Hermitian Einstein manifold $(M,g,J)$ such that $(g,J)$ is not Bismut Torsion Parallel?
\end{question}
On the examples $(M_f, I, h)$ we also investigate the existence of generalized K\"ahler structures.  Let  $I_-$  be the complex structure  $I$ on  the mapping torus $M_f$.  In Theorem \ref{theorem:kahlermappingtorus} we prove that   $(M_f,  I_-, h)$ admits another complex structure $I_+$ compatible with $h$ and such that $d^c_+\omega_+=-d^c_-\omega_-$, i.e., $(h,I_\pm)$ defines a generalized K\"ahler structure on $M_f$. 
In the last Section we also give another description of  the mapping tori $M_f$ as total spaces of holomorphic fibrations with fibre $K$ over the Hopf Surface. To such spaces it is always possible to associate the Borel spectral sequence  (Theorem \ref{theorem:holfib}), which relates the Dolbeault cohomology of $M_f$ with the ones of $K$ and $\Sph^3 \times \Sph^1$.  \\
Since the mapping tori $M_f$ topologically corresponds to $\Z$-quotient of $K \times \Sph^3 \times \R$, i.e., $\Z$-quotient of a product of a K\"ahler manifold with the Samelson space $\Sph^3 \times \R$, taking inspiration from the local Samelson space construction, we provide a generalization of the previous machinery to $\Z^n$-quotients of the product manifolds $K \times G'$, where $G'=G \times \R^n$ is any Samelson space (for a more detailed construction see  Lemma \ref{lemma:pluriclosed2} and Proposition \ref{prop:flatsam}). \\
As in the case of mapping tori, these quotients admit a generalized K\"ahler structure (Theorem \ref{theorem:samspaces}). Moreover, in the case of $G'=\Sph^3 \times \R$ the generalized K\"ahler structure constructed is different from the one constructed in Theorem \ref{theorem:kahlermappingtorus}. As we point out, in the first case the complex structures of generalized K\"ahler metric induce opposite orientations, whereas in the second case the complex structures of the Generalized K\"ahler metric induce the same orientation.

\section{Universal Cover of Bismut Hermitian Einstein manifolds with parallel torsion} \label{section:one}


In this section we obtain a characterization of the universal cover of compact complex manifolds admitting a SKT and CYT metric whose Bismut connection has parallel torsion. \\
Before stating the main result of the section we recall some preliminary definitions and known results which will be useful in the paper. 
 Given  a  Hermitian manifold.  $(M,g,J)$,  a  connection $\nabla$ on $TM$ is said to be Hermitian if $\nabla J=0$ and $\nabla g=0$. 
 
In (\cite{PG}), Gauduchon proved that there exists an affine line of canonical Hermitian connections, passing through the Chern connection and the Bismut connection, which are completely determined by their torsion. In this paper we are mainly interested in the latter. Firstly introduced by Bismut in \cite{Bi}, the Bismut connection is the unique Hermitian connection having totally skew-symmetric torsion, i.e., it is the unique Hermitian connection such that  $H (\cdot, \cdot, \cdot) =g(T^B (\cdot, \cdot), \cdot)$ is a $3$-form on $M$. 
We recall the following 

\begin{definition}

a Hermitian manifold $(M,g,J)$ is said to be

\begin{enumerate}

\item  \emph{Bismut torsion-parallel}, or BTP in short, if $\nabla^B H=0$;

\item SKT or  \emph{pluriclosed}  if  $dH=0$, or, equivalently, $dd^c\omega=0$;

%
%

\item  \emph{Bismut K\"ahler-like} (BKL in short) if the Bismut curvature  $R^B$ satisfies the first Bianchi identity and the type condition, i.e., $R^B (X,Y,Z,W)=R^B (JX,JY,Z,W)$, for every vector fields $X, Y, Z, W$ on M (we write $X,Y,Z,W\in\mathfrak{X}(M)$).

\end{enumerate}

\end{definition}

In \cite{ZZ3}, Zhao and Zheng proved that  a Hermitian  manifold is BKL if and only if it is  BTP and SKT. More precisely, the following Theorem holds:
\begin{theorem}[\cite{ZZ3}] \label{theorem:zz3}
a Hermitian manifold $(M,g,J)$ is  BKL if and only if $\nabla^B$ has parallel Bismut torsion and $(g,J)$ is SKT.
\end{theorem}
The relations between the first Bianchi identity for the Bismut connection, the SKT and the BTP conditions have been investigated by the second author and Tardini in \cite{FT}.  In particular, they proved the following 
\begin{theorem} [\cite{FT}] \label{thm:finotardini}
Let M be a complex manifold with a  compatible SKT metric $g$ such that the Bismut connection satisfies the first Bianchi identity. Then $\nabla^B H =0$.
\end{theorem}
While in the Riemannian case there is only one natural trace of the Curvature tensor (yielding the Ricci tensor), in the complex setting many different traces are possible. 
\begin{definition}
Let $(M,g,J)$ be a Hermitian manifold. Let 
\[
\rho^B (X,Y)=\frac{1}{2} \  \sum_{i=1}^{2n} g(R^B(X,Y) Je_i, e_i),
\]
where $\{e_i\}$ is a  orthonormal frame of the tangent space of $M$ at a given point. The tensor $\rho^B$ is also  known in literature as the (first) Bismut Ricci curvature. 
\end{definition}
\begin{definition}
Let $(M,J, g)$ be a Hermitian manifold. If $\rho^B=0$ then the Hermitian  structure $(J, g)$ is said to be Calabi-Yau with torsion (CYT in short).
\end{definition}
When a Hermitian structure $(g,J)$ is both CYT and SKT, it is called Bismut Hermitian-Einstein (BHE in short). The main purpose of this section is to investigate the universal cover of  a BHE manifold with parallel Bismut torsion. One of the main tool of our proof is given by the following Theorem of Zhao and Zheng.
\begin{theorem}[\cite{ZZ}]
Let $(M,g,J)$ be a compact BKL Hermitian manifold without any K\"ahler de Rham factor. If the (first) Bismut Ricci curvature vanishes, then $(g,J)$ is Bismut flat. 
\end{theorem}
\begin{remark}
By Theorem \ref{theorem:zz3}, the latter Theorem could be equivalently reformulated in the case of $(M,g,J)$ being a compact BHE and BTP manifold without any K\"ahler de Rham factor.
\end{remark}
\begin{remark} \label{remark:completeness}
The latter Theorem holds with the same proof in the case of the metric $g$ being complete (and the manifold $M$  not necessarily compact).
\end{remark}
\smallskip

\smallskip

\begin{lemma}
Let $(M, h, I)$ be a complete Hermitian manifold with parallel Bismut torsion $H$, i.e., $\nabla^B H=0$. Let 
\[
  \cal{K}:=\ker H=\{ X \in \mathfrak{X}(M)\ | \ \iota_X H =0\}.
\] 
Then, $\cal{K}$ is an integrable distribution of $M$ preserved by both the Levi-Civita  $\nabla^{LC}$ and the Bismut connection $\nabla^B$.
\end{lemma}
\begin{proof}

 Assume $\cal{K} \neq 0$ and consider an open set $U$ in which its dimension is constant. Let $X$ be a local vector field in $\cal{K}$ and $Y,Z,W$ any vector fields on $M$.  Notice that $$\nabla^{LC}_Y Z = \nabla^B_Y Z - T^B(Y, Z)$$ and $\cal{K} = {\mbox {Ker}} (Y\rightarrow T^B(Y,Z)).$
From these assumptions we have $$0=(\nabla^B_W H)(X,Y,Z) = -H(\nabla^B_W X, Y, Z),$$ so $\nabla^B_W X \in \cal{K}$ and $\nabla^B_W X = \nabla^{LC}_W X \in \cal{K}.$ This means that $\cal{K}$ is parallel with respect to both connections. Then parallel transport along geodesics of an orthonormal frame shows that it has constant dimension. The integrability follows from the fact that it is parallel with respect to the Levi-Civita connection
 
\end{proof}
\vspace{.2in}

Using the previous Lemma, in the case of $(M, I, h)$ being a compact non-K\"ahler complex manifold with Bismut parallel torsion $H$ we may orthogonally decompose $TM$ as the orthogonal direct sum
\[
TM=\cal{K}\oplus\cal{K}^{\perp},
\]
where $\cal{K}^{\perp}$ is the orthogonal distribution of $\cal{K}$. Since the Levi-Civita and the Bismut connections are  both metric, it then follows that $\cal{K}^{\perp}$ is preserved by both $\nabla^B$ and $\nabla^{LC}$.


It is straightforward to observe that since $(I, h)$ is non-K\"ahler, $\cal{K}^{\perp}$ is not trivial. \\
Let $\cal{F}_1$ be the maximal sub-distribution of $\cal{K}$ preserved by both $\nabla^{LC}=\nabla^{B}$  and by the complex structure $I$, i.e. $I (\cal{F}_1)=\cal{F}_1$.
Then, we may further decompose $\cal{K}$ as the direct sum
\[
\cal{K}=\cal{F}_1 \oplus \cal{W},
\]
where $\cal{W}$ is the orthogonal complement of $\cal{F}_1$ inside $\cal{K}$ with respect to the restricted metric $h\restrict{\cal{K}}$. 

The tangent bundle isometrically splits as the direct sum of three mutually orthogonal subbundles
\[
TM=\cal{K}\oplus\cal{K}^{\perp}= \cal{F}_1 \oplus \cal{W} \oplus \cal{K}^{\perp}= \cal{F}_1 \oplus \cal{F}_2,
\]
where $\cal{F}_1$ and $\cal{F}_2=\cal{W} \oplus \cal{K}^{\perp}$ are preserved by both $\nabla^{B}$ and $\nabla^{LC}$ and are also $I$-invariant.

We mention also the following obvious:  
\begin{lemma} \label{lemma:derhamfactor}
Let $(M, h, I)$ be a compact non-K\"ahler complex manifold with parallel Bismut torsion $H$. Consider the orthogonal splitting 
\[
TM=\cal{F}_1 \oplus \cal{F}_2,
\]
described above. Then $\cal{F}_2$ does not contain any (non-trivial) sub-distribution $\cal{F}_3\subset\cal{F}_2$ such that $\cal{F}_3\subset\cal{K}$, $\nabla^{LC}\cal{F}_3\subset\cal{F}_3$ and $I(\cal{F}_3)=\cal{F}_3$.
\end{lemma}

Now notice that Theorem 1.1 is trivially true when the metric $h$ is K\"ahler. Then for its proof we only need the following:


\begin{theorem} \label{theorem:universalcovering}
Let $(M, I)$ be a compact  non-K\"ahler complex manifold admitting a SKT and CYT  $I$-Hermitian metric  $h$ such that its  Bismut torsion $3$-form $H$  is parallel, i.e. $\nabla^B H=0$. Then, the   Riemannian holomorphic universal cover $(\tilde M,\tilde I,  \tilde h)$  of $(M, I, h)$  is holomorphically isometric to the product $(M_1,  J_1, g_1) \times (M_2,  J_2, g_2)$, where $(M_1, J_1, g_1)$ is a K\"ahler Ricci flat manifold and $(M_2, J_2,  g_2)$ is a Samelson space. \end{theorem}
\begin{proof}
Let $\tilde M$ be the universal holomorphic Riemannian cover of $M$. By a standard argument $\tilde M$ is complete. Then the Bismut Hermitian Einstein structure $(h,I)$ of $M$ with parallel torsion $H$ lifts to a structure $(\tilde h,\tilde I,\tilde H)$ of the same kind on  the cover $\tilde M$. By the de Rham splitting Theorem, $(\tilde M, \tilde h, \tilde I)$ is holomorphically isometric to a product $ (M_1,  J_1, g_1)\times(M_2,  J_2, g_2)$ such that the decomposition $TM= \cal{F}_1 \oplus \cal{F}_2$ pulls-back to the decomposition $T\tilde M= TM_1 \oplus TM_2$. This follows from the fact that the two distributions $ \cal{F}_1$ and $\cal{F}_2$ are $I$-invariant and that, by construction, the complex structures on the distributions  pull-back to complex structures $J_1$ and $J_2$ on $TM_1$ and $TM_2$, respectively. Moreover, as $\tilde h$ is complete, also $h_1$ and $h_2$ are complete.
The splitting of the metric $\tilde h$ and the complex structure $\tilde I$ gives a splitting of the torsion $\tilde H$ as $H_1 + H_2$. Moreover, since $\cal{F}_1$ is in the kernel distribution, we must have $H_1 \equiv 0$, and, so, $ (M_1, g_1, J_1)$ is a K\"ahler Ricci-flat factor. Furthermore  $(M_2,h_2, J_2)$ is a Bismut Hermitian-Einstein structure with parallel Bismut-torsion and, by Lemma \ref{lemma:derhamfactor}, $(M_2,h_2, J_2)$ does not contain any K\"ahler de Rham factor. Hence, it must be Bismut flat by Theorem \ref{theorem:zz3} stated in the complete case. The last statement follows by the characterization of simply connected Bismut flat manifolds given in \cite[Theorem 5]{WYZ} and by the fact that the metric $h_2$ is complete. 
\end{proof}
\begin{remark}
If $\dim_{\C}(\cal{F}_1)\le 1$, then $M$ is Bismut flat. Indeed, in such a case, $M$ does not have any K\"ahler de Rham factor of dimension bigger than $1$ and so it is Bismut flat, as follows by \cite[Theorem 3]{ZZ}.
\end{remark}
\section{A construction  via mapping tori} \label{section:two}
Given a smooth manifold $M$ and a diffeomorphism $f$ of M, the mapping torus (or suspension) of $f$ is defined to be the quotient $M_f$ of the product $M \times \R$ by the $\Z$-action defined by
\[(p, t) \mapsto (f^n(p), t + n).\]
Topologically, a mapping torus is a fiber bundle over $\Sph^1$ via the natural projection $\pi : M_f \to \Sph^1$ defined by $(p, t) \mapsto e^{2\pi it}$.\\
In this section we construct a SKT metric  on the mapping torus of a product of a K\"ahler manifold  $(K,J,g)$ and the $3$-sphere $\Sph^3$ via a diagonal diffeomorphism $f=(\psi,Id_{\Sph^3})$, where $\psi$ is a K\"ahler isometry of $(K, J, g)$ (Lemma \ref{lemma:pluriclosed}). In Proposition \ref{prop:flathopf} we show that, if $g$ is Ricci flat, then the SKT structure constructed on $M_f$ is also CYT.  In particular, if $g$ is not flat, then the SKT and CYT metric is not Bismut flat. 
\begin{lemma} \label{lemma:pluriclosed}
Let $(K,J,g)$ be a compact K\"ahler manifold of complex dimension $k$ and let $\psi$ be a K\"ahler isometry, i.e., $\psi$ is an holomorphic diffeomorphism of $K$ satisfying $\psi^*(g)=g$. Then, the mapping torus 
\[
M_f=(K \times \mathbb{S}^3)_f,
\]
with $f=(\psi,Id_{\Sph^3})$, admits a SKT structure $(I, h)$.
\end{lemma}
\begin{proof}
Let us fix on $K \times \C^2 \setminus \{(0,0)\}$ the product complex structure $J \times J_-$, where $J_-$ is the standard complex structure on $\C^2 \setminus \{(0,0)\}$.\\
Consider the following free and proper discontinuous $\Z$-action on $K \times \C^2 \setminus \{(0,0)\}$
\begin{equation} \label{eqn:Zaction}
 n \cdot (p, \underline{x}) \mapsto (\psi^n (p), 2^n \underline{x}).
\end{equation}
The associated automorphisms $\phi_n$ are manifestly holomorphic with respect to $J \times J_-$, and so, $J \times J_-$ descends to a complex structure $I$ on  the quotient $\bigslant{K\times \C^2\setminus\{(0,0)\} }{ \Z}$.\\
We claim that $\bigslant{K\times \C^2\setminus\{(0,0)\} }{ \Z}$ is diffeomorphic to $M_f=(K \times \mathbb{S}^3)_f$. Let
$$
\alpha: K \times \Sph^3 \times \R \to K \times \C^2 \setminus \{(0,0)\}, \quad 
           (p, q, t)  \mapsto (p, 2^t \cdot q)
$$
with inverse
$$
\alpha^{-1}: K \times \C^2 \setminus \{(0,0)\} \to K \times \Sph^3 \times \R, \quad 
           (p, \underline{x}) \mapsto (p, \frac{\underline{x}}{\norm{\underline{x}}}, \log_2 \norm{\underline{x}}).
$$
The induced maps on the quotient \[\alpha: M_f \to K \times \C^2 \setminus \{(0,0)\}/ \Z, \quad \alpha^{-1}: \bigslant{K\times \C^2\setminus\{(0,0)\} }{ \Z} \to M_f,\] are well defined and give the claimed identification $\bigslant{K\times \C^2\setminus\{(0,0)\} }{ \Z}  \cong M_f$. Therefore, $M_f$ inherits from $\bigslant{K\times \C^2\setminus\{(0,0)\} }{ \Z}$ the complex structure $I$.\\
Consider
$\Omega=\omega+\omega_-,
$
where in the coordinates $(z_1,z_2)$ on $\C^2 \setminus \{(0,0)\}$ 
\begin{equation} \label{eqn:omegaminus}
\omega_-= \frac{i}{R^2} (dz_1 \wedge d \overline{z}_1 +  dz_2 \wedge d \overline{z}_2),
\end{equation}
with $R^2=z_1 \overline{z}_1+z_2 \overline{z}_2$. \\
Observe that $\Omega$ is a well defined non-degenerate global $2$-form on $\bigslant{K\times \C^2\setminus\{(0,0)\} }{ \Z} \cong M_f$, as it is preserved by $\phi_n^*$, for each $n$ in $\Z$. Furthermore, $\Omega$ is of type $(1,1)$ with respect to $I$.\\
The Hermitian metric $h= \Omega  I$ is hence given by
\begin{equation} \label{eqn:gkmetric}
h=g+\frac{1}{R^2} (dz_1d \overline{z}_1 +  dz_2  d \overline{z}_2).
\end{equation}
We compute $d^c \Omega$, where $d^c=-I dI$, and $dd^c \Omega$. 
\begin{align}
d^c  \Omega =&- \frac{1}{R^4}\left[(\overline{z}_2dz_2-z_2 d\overline{z}_2)\wedge dz_1 \wedge d\overline{ z}_1+ (\overline{z}_1dz_1-z_1 d\overline{z}_1)\wedge dz_2 \wedge d\overline{ z}_2\right], \label{eqn:H} \\
d(d^c \Omega)=&d\left( \frac{1}{R^4}\left[(\overline{z}_2dz_2-z_2 d\overline{z}_2)\wedge dz_1 \wedge d\overline{ z}_1+ (\overline{z}^1dz_1-z_1 d\overline{z}_1)\wedge dz_2 \wedge d\overline{ z}_2\right] \right) \nonumber 
= 0 \nonumber.
\end{align}
Therefore, $(I, h)$ is a SKT structure on $M_f$.
\end{proof}
We investigate some cohomological properties of $M_f$. In the next proposition we prove that  $M_f$ is formal (in the sense of Sullivan) and non-K\"ahler. 
\begin{proposition} \label{proposition:prop}
Let $M_f=(K \times \mathbb{S}^3)_f$ be the mapping torus  constructed as in Lemma \ref{lemma:pluriclosed}. Then
\begin{enumerate}
\item $M_f$ is diffeomorphic to the product of $M_\psi \times \Sph^3$, where $M_\psi$ is the mapping torus of the K\"ahler manifold $K$ with respect to the K\"ahler isometry $\psi$; \\
\item $M_f$ is formal; \\
\item $M_f$ is non-K\"ahler.
\end{enumerate}
\end{proposition}
\begin{proof} 
Consider the map
\[
\beta: M_f  \to M_\psi \times \Sph^3, \quad 
 [(p,q,t)]  \mapsto ([(p,t)],q)
\]
with inverse 
\[
\beta^{-1}: M_\psi \times \Sph^3 \to M_f, \quad 
( [(p,t)],q) \mapsto [(p,q,t)] .
\]
It is immediate to observe that $\beta$ and $\beta^{-1}$ are well defined diffeomorphisms. Indeed,
\[
\begin{split}
\beta [(p,q,0)]&=([(p,0)],q)=([(\psi(p),1],q)=\beta [(\psi(p),q,1)], \\
\beta^{-1} ( [(p,0)],q)&=[(p,q,0)]= [(\psi(p),q,1)]=\beta^{-1} ( [(\psi(p),1)],q).
\end{split}
\]
The second statement follows by the identification $M_f \cong M_\psi \times \Sph^3$ proved above. In fact, by \cite{Li}, the mapping tori of compact K\"ahler manifolds are compact co-K\"ahler manifolds and so they are formal in the sense of Sullivan (see for instance \cite{CLM}). Since $\Sph^3$ is formal, $M_f $ is the product of formal manifolds and, hence, formal.\\
To prove that $M_f$ is non-K\"ahler is suffices to show that the first Betti number $b_1(M_f)=b_1(M_\psi)$ is odd. \\ 
The cohomology groups of mapping tori can be computed as follows
\[
              H^r (M_\psi)= N^r \oplus C^{r-1}, 
\]
where $N^r:= \ker (\psi_r^*-Id)$ and $C^{r-1}:= \operatorname{coker} (\psi_{r-1}^*-Id)$, where $\psi_r^*-Id$ is the map induced by $\psi$ at the level of the r-th cohomology groups. \\
Clearly, 
\[
 H^0 (M_\psi)= K^0 =C^0=\R,
\]
and therefore $H^1 (M_\psi)= N^1 \oplus C^{0}$. Hence, $b_1(M_\psi)=\dim (N^1)+1$. We claim that $\dim (N^1)$ is even. \\
Let $[\alpha]$ be in $\ker(\psi_1^*-Id)$. Without loss of generality, we may assume that $\alpha$ is the harmonic representative of its cohomology class, as $K$ is compact by hypothesis. As $[\alpha]$ is in  $\ker(\psi_1^*-Id)$, $\psi_1^*(\alpha)=\alpha + d \eta$ for some smooth function $\eta$ on $K$. Let us denote by $( \cdot, \cdot)$ the $L^2$ product defined on $(K,g)$ where $g$ is the K\"ahler metric  on $K$.  We want $d\eta$ to be zero. Indeed,
\[
\begin{split}
0&=(\psi^*(\Delta \alpha), d \eta)=(\Delta(\psi^*\alpha), d \eta)=(\Delta(\alpha+d\eta), d \eta)=\\
&=(\Delta(d\eta), d \eta)=(dd^*d\eta, d\eta)=(d^*d\eta, d^*d\eta)= \norm{d^*d\eta}^2,\\
\end{split}
\]
where the second equality holds since the pullback by an isometry commutes with the laplacian operator.\\
Since $d^*d\eta=0$, then $\Delta(d\eta)=0$. Hence, $\alpha$ and $\alpha+d\eta$ are two harmonic representatives of the same cohomology class. By uniqueness, $d\eta=0$.\\
We just proved that, if $[\alpha]$ is in the $\ker(\psi_1^*-Id)$ and $\alpha$ is the harmonic representative, then $\psi^*\alpha=\alpha$. \\
Recall that on K\"ahler manifolds, $J$ induces a map
$$
    J: \cal{H}^1(K) \to \cal{H}^1(K),  \quad \gamma \mapsto J\gamma,
$$
where $J \gamma(X)=\gamma (JX)$ and $\cal{H}^1(K)$ is the vector space of harmonic $1$-forms. Observe that the latter map is well defined since on K\"ahler manifolds $J$ commutes with the laplacian operator. \\ We claim that if $\alpha$ is the harmonic representative of $[\alpha] \in \ker(\psi^*_1-Id)$, then also $[J\alpha]$ is in $\ker(\psi^*_1-Id)$. In fact, if we apply $\psi^*$ to $J\alpha$ we get
\[
   \psi^*(J\alpha)=J(\psi^* \alpha)=J\alpha.
\]
Then the dimension of $N^1$ must be even. 
\end{proof}

\begin{proposition} \label{prop:flathopf}
Let $M_f$ be the mapping torus $M_f=(K \times \mathbb{S}^3)_f$ constructed as in Lemma \ref{lemma:pluriclosed}. If  the K\"ahler metric  $g$  on $(K, J)$ is (non-flat)  Ricci flat, then  the  Bismut connection associated to the SKT metric $h$  on $(M_f, I)$,  constructed in Lemma \ref{lemma:pluriclosed} is  CYT  with non-flat Bismut connection. 
\end{proposition}
\begin{proof}
Consider the product metric  $\tilde {g}: = g+ \frac{g_E}{R^2}$ on $K \times \C^2 \setminus \{0\}$, with $g$ being the K\"ahler metric of $K$, $g_E$ being the Euclidean metric of $\C^2$ and $R^2= z_1 \overline{z}_1 +  z_2 \overline{z}_2$, where $(z_1,z_2)$ are  the standard coordinates of $\C^2$.  \\
Since   $(J,g, \omega)$ is K\"ahler, the Bismut connection of $(J, g, \omega)$  is the Levi-Civita  $\nabla^{LC}_1$.
Analogously, we denote by  $\nabla^B_{2}$ the Bismut connection of the Hermitian structure $(\frac{g_E}{R^2},J_-,\omega_-)$ on $ \C^2 \setminus \{0\}$. $\nabla^B_{2}$ is defined by the relation
\[
    \frac{g_E}{R^2}( (\nabla^B_{2})_{X} Y,Z ) = \frac{g_E}{R^2}((\nabla^{LC}_{2})_{X}Y,Z)- \frac{1}{2} d^c_{J_{-}}\omega_-(X,Y,Z),
\]
for any $X,Y,Z \in \mathfrak{X}(\C^2 \setminus \{0\})$.  In the following, we will denote by $H_2$ the torsion $3$-form $- d^c_{J_{-}}\omega_-$, which is actually closed, as already computed in the proof of Lemma \ref{lemma:pluriclosed}. \\
It is straightforward to observe that the fundamental form $\tilde \omega=\omega+\omega_-$ of the Hermitian manifold $(K \times \C^2 \setminus \{0\}, \tilde{J}:=J \times J_-, \tilde {g})$ satisfies $$d^c_{\tilde {J}}( \omega+\omega_-)=d^c_{J_-}\omega_-=-H_2.$$
Consider  the  unique metric connection  $\tilde {\nabla}$ on $(K \times \C^2 \setminus \{0\}, \tilde g, \tilde J)$ with skew-symmetric torsion $H_2$ determined by
\[
    \tilde g (\tilde {\nabla}_{X}Y,Z)=  \tilde g  (\tilde {\nabla}^{LC}_{X}Y,Z)- \frac{1}{2} d^c_{J_{-}}\omega_-(X,Y,Z),  \ \forall X,Y,Z \in \mathfrak{X}(K \times \C^2 \setminus\{0\} ), 
\]
where $\tilde \nabla^{LC}$ is the Levi-Civita connection of the product metric $ \tilde g = g+ \frac{g_E}{R^2}$.  It turns out that $\tilde \nabla$ is the Bismut connection of  $(K \times \C^2 \setminus \{0\}, \tilde J, \tilde g)$, since  $d^c_{\tilde J} \tilde \omega = d^c_{J_-}\omega_-$, as already remarked. Hence, from now on we set $\tilde \nabla=\tilde \nabla^B$ and we will denote by $\tilde H$ the torsion $3$-form $-d^c_{\tilde J} \tilde \omega$ of $\tilde \nabla^B$. \\
Let now compute  the Bismut curvature tensor $\tilde {R}^B$ . 
When the torsion in a closed $3$-form, we have a general formula to compute the Bismut curvature tensor $\tilde R^B$ in terms of the Riemannian curvature tensor $\tilde R^{LC}$ (see for instance \cite[Proposition 3.18]{GFS}). \\
For any quadruple $(X,Y,Z,W)$  of vector fields on $K \times \C^2 \setminus \{0\}$, it holds that:
 \begin{align*}
 \tilde {R}^B (X,Y,Z,W) =& \quad\tilde {R}^{LC} (X,Y,Z,W) 
+ \frac{1}{2} \tilde \nabla^{LC}_{X}\tilde H(Y,Z,W) - \frac{1}{2} \tilde \nabla^{LC}_{Y}\tilde H(X,Z,W) \\[4pt]
&- \frac{1}{4}\ \tilde g (\tilde H(X,W),\tilde H(Y,Z))  +\frac{1}{4} \ \tilde g (\tilde H(Y,W),\tilde H(X,Z)).
 \end{align*}
Since $\mathfrak{X}(K \times \C^2 \setminus \{0\}) \cong \mathfrak{X}(K) \oplus \mathfrak{X}(\C^2 \setminus \{0\})$ as $C^{\infty}(K \times \C^2 \setminus \{0\})$ module, by the $C^{\infty}(K \times \C^2 \setminus \{0\})$-multi linearity of the curvature tensor we may compute $\tilde R^B$ on vector fields of the kind $X_1+X_2$, where $X_1 \in  \mathfrak{X}(K) $ and $X_2 \in  \mathfrak{X}(\C^2 \setminus \{0\})$. Furthermore, by construction, $\tilde H= H_2$. Therefore  $\tilde R^B (X_1+X_2,Y_1+Y_2, Z_1+Z_2,W_1+W_2)$ is given by 
 
 \begin{align*}
\tilde {R}^B(X_1+X_2,Y_1+Y_2,Z_1+Z_2,W_1+W_2) =&\quad \tilde R^{LC}(X_1+X_2,Y_1+Y_2,Z_1+Z_2,W_1+W_2) \\[4pt]
&+ \frac{1}{2} \tilde \nabla^{LC}_{X_1+X_2}\tilde H(Y_1+Y_2,Z_1+Z_2,W_1+W_2)\\[4pt]
&- \frac{1}{2} \tilde \nabla^{LC}_{Y_1+Y_2}\tilde H(X_1+X_2,Z_1+Z_2,W_1+W_2) \\[4pt]
&- \frac{1}{4} \tilde g(\tilde H(X_1+X_2,W_1+W_2),\tilde H(Y_1+Y_2,Z_1+Z_2)) \\[4pt]
&+\frac{1}{4} \tilde g (\tilde H(Y_1+Y_2,W_1+W_2),\tilde H(X_1+X_2,Z_1+Z_2)),
 \end{align*}
  which reduces  to 
 \begin{align*}
\tilde {R}^B(X_1+X_2,Y_1+Y_2,Z_1+Z_2,W_1+W_2) =& \quad  R^{LC}_{1}(X_1,Y_1,Z_1,W_1)+R^{LC}_{2}(X_2,Y_2,Z_2,W_2) \\
&+ \frac{1}{2} \left(\nabla^{LC}_ 2\right)_ {X_2}H_2(Y_2,Z_2,W_2)\\
&- \frac{1}{2} \left(\nabla^{LC}_2\right)_{Y_2}H_2(X_2,Z_2,W_2) \\
& - \frac{1}{4} \frac{g_E}{R^2}(H_2(X_2,W_2),H_2(Y_2,Z_2))\\
&  +\frac{1}{4} \frac{g_E}{R^2}(H_2(Y_2,W_2),H_2(X_2,Z_2)).
  \end{align*}
  Therefore 
\begin{align*}
\tilde R^B (X_1+X_2,Y_1+Y_2, Z_1+Z_2,W_1+W_2) =&R^{LC}_{1}(X_1,Y_1,Z_1,W_1)+R^B_{2}(X_2,Y_2,Z_2,W_2) \\
 =& R^{LC}_{1}(X_1,Y_1,Z_1,W_1),
 \end{align*}
where the last equality follows  from the fact that the Bismut curvature tensor $R^B_{2}$ of $(\frac{g_E}{R^2}, J_-, H_2)$ vanishes (for a reference see for instance \cite[Lemma 3]{WYZ}).\\
We are now ready to compute the Ricci form $\tilde \rho^B$ of the Bismut connection $\tilde \nabla^B$. As already done before, without loss of generality, we may compute $\tilde \rho^B$ on decomposable vector fields. By definition 
\[
\begin{split}
\tilde \rho^B (X_1+X_2,Y_1+Y_2)=& \frac{1}{2}  \sum_{i=1}^{2k+4} \tilde R^B (X_1+X_2,Y_1+Y_2,J e_i ,e_i)\\
=& \frac{1}{2} \sum_{i=1}^{2k} R^B_1 (X_1,Y_1,J e_i ,e_i) =  \rho^{LC}_{1} (X_1,Y_1) =0,
\end{split}            
\]
where $  \rho^{LC}_{1}$ is the Ricci form of $(K,g)$,\ $\{e_1,\dots,e_{2k}\}$ is a local orthonormal basis of $TK$, and $\{e_{2k+1},\dots,e_{2k+4} \}$ is a local orthonormal basis of $T\C^2 \setminus\{0\}$. 	\\
For each $n \in \Z$, the diffeomorphism $\phi_n$ is a holomorphic isometry of the product metric $\tilde g$, with respect to the product complex structure $\tilde J$. This suffices to show that also the torsion form is preserved by $\phi_n^*$, as $\tilde H=\tilde J d \tilde \omega$ .\\
In particular, the triple $( \tilde g , \tilde J , \tilde H)$ descends to the mapping torus $M_f$ to $(h, I, -d^c \Omega)$, where $(h, I, \Omega)$ is the SKT structure constructed in Lemma \ref{lemma:pluriclosed}. We then have that  $(M_f, h, I)$ is a SKT and CYT manifold with non-flat  Bismut connection.
\end{proof}
\begin{remark} \label{remark:remark}
Observe that the torsion $H=-d^c \Omega$ is always parallel with respect to the Bismut connection. Since the Bismut curvature tensor satisfies the First Bianchi Identity and the Hermitian structure is SKT, the statement follows by Theorem \ref{thm:finotardini}.
\end{remark}
\begin{remark}
Let $K$ be a $K3$ surface. Then the Riemannian holomorphic universal cover of $(M_f=(K \times \Sph^3)_f,  I, h)$ is given by $(K \times \Sph^3 \times \R, \tilde g, \tilde   J)$, where $\tilde g$ is the product metric $g + \frac{g_E}{R^2} $ and $\tilde J$ is the product complex structure $ J \times J_-$. In the notation of the Theorem \ref{theorem:universalcovering}, $G= \Sph^3 \times \R$.
\end{remark}
By Proposition \ref{prop:flathopf}, we are mainly interested in the case of $K$ being a $K3$ surface. In particular, we want to show that when $K$ is a $K3$ surface, the CYT and SKT mapping tori $M_f$ are not trivial, i.e., they do not split to a product of a Ricci flat K\"ahler manifold with a Bismut flat one. To do so, we briefly recall the following renowned result in the Theory of $K3$ surfaces.
\begin{theorem}[Torelli Theorem \cite {BR, PS}]\label{theorem:torelli}
Let $K,K'$ be $K3$ surfaces and let $\Omega_K, \Omega_{K'}$ be nowhere vanishing holomorphic $2$-forms on $K$ and $K'$, respectively. Assume that there exists an isometry of lattices
\[
\alpha:H^2(K,\Z) \to H^2(K',\Z) 
\]
satisfying \\
\begin{enumerate}
\item $\alpha ([\Omega_k])=c \cdot [\Omega_{K'}]$, for some $c \in \C^*$, \\
\item $\alpha$ sends a K\"ahler class of $K$ to a K\"ahler class of $K'$. \\
\end{enumerate}
Then there exists a unique isomorphism of $K3$ surfaces $g:K'\to K$ such that $g^*=\alpha$ on $H^2(K,\Z)$.
\end{theorem} 
We use the previous Theorem to show the following Proposition.
\begin{proposition} \label{proposition:K3}
Let $K$ be a $K3$ surface admitting an automorphism $\psi \neq Id_K$. Then $\dim(N^2_\psi)<22$, where $N^2_\psi=\ker(\psi_2^*-Id)$. 
\end{proposition}
\begin{proof}
By contradiction, assume that $\dim(N^2_\psi)=22$, i.e., $\psi^*=Id$ on $H^2(K,\R)$. Then, the restriction $\psi^*\restrict{H^2(K,\Z)}=Id_{H^2(K,\Z)}$ satisfies the hypothesis of Theorem \ref{theorem:torelli}. Hence, there exists a unique automorphism $g$ of $K$ such that $g^*\restrict{H^2(K,\Z)}=Id_{H^2(K,\Z)}$. By uniqueness, $g=Id_K$. \\
Moreover, since $\psi$ is another automorphism of $K$ which restricted to $H^2(K,\Z)$ is the identity, we must have $\psi=Id_K$. The contradiction follows.
\end{proof}
\begin{corollary} \label{corollary:nonproduct}
Assume that $K$ is a $K3$ surface admitting a non-trivial K\"ahler isometry $\psi$, i.e., $\psi\neq Id_K$. Then the mapping torus $M_f=(K \times \Sph^3)_{(\psi,Id_{\Sph^3})}$ constructed in Lemma \ref{lemma:pluriclosed} is never trivial.
\end{corollary}
\begin{proof}
By the K\"unneth formula, $H^2(K \times \Sph^3 \times \Sph^1, \R)\cong H^2(K,\Z)\cong \R^{22}$. \\
We claim that $\dim(H^2(M_f,\R))<22$.  Using Proposition \ref{proposition:prop}, $M_f \cong M_\psi \times \Sph^3$ and so, again by K\"unneth formula, $H^2(M_f, \R)\cong H^2(M_\psi,\Z)\cong N^2_\psi \oplus C^1_\psi$, where $N^2_\psi=\ker(\psi_2^*-Id)$ and $C^1_\psi=\operatorname{coker}(\psi_1^*-Id)$.  Since $\psi$ is an automorphism of $K$ different from the identity, $\dim( N^2_\psi)<22$ by Proposition \ref{proposition:K3}, and $\dim(C^1_\psi)=0$ as $H^1(K,\R)=0$.  It then follows that $\dim(H^2(M_f, \R))<22$, concluding the proof. 
\end{proof}
\begin{remark}
Since any normal subgroup of $\pi_1(M_f) \cong \Z$ of finite index is of the kind $k\Z$ for some integer $ k > 1$, it follows that any finite cover of $M_f=(K \times \Sph^3)_{(\psi,Id_{\Sph^3})}$ corresponds to a quotient of $K \times \C^2 \setminus \{(0,0)\}$ by the action of $k\Z$ given by
\[
kn \cdot (p, \underline{x}) \mapsto (\psi^{kn} p, 2^{kn} \underline{x}).
\]
Moreover, since $\psi$ has finite order, there exists a $\overline{k} > 1 $ such that $\psi^{\overline{k}}=Id$, i.e., there exists a finite cover of $M_f$ which splits as a product of $K$ and the Hopf surface $\Sph^3 \times \Sph^1$.
\end{remark}


\section{A generalization of local Samelson spaces}

Let $(G'=G \times \R^n, g'=b+g_E, J_L)$ be a Samelson space, i.e., $G$ is a compact, connected and simply connected semisimple Lie group, $g'=b+g_E$ is the bi-invariant metric on $G'$ given by the product of the bi-invariant metric $b$ on $G$ and the euclidean metric $g_E$, and $J_L$ is a left invariant complex structure compatible with $g'$. We assume that $\dim(G) \ge 3$ and $\dim(G')=2r$.\\
Let $\rho$ be a group homomorphism $\rho: \Z^n \to {\mbox {Isom}}(G)$. Then, for each $m \in \Z^n$ the diffeomorphism $m \cdot (q,t)= (\rho(m)q, t+m)$ is an isometry of $g'$.\\
 Furthermore, consider a compact K\"ahler manifold $(K,g,J)$ and a group homomorphism $\psi: \Z^n \to {\mbox {Isom}}_{hol}(K)$, where ${\mbox{Isom}}_{hol}(K)$ is the group of holomorphic K\"ahler isometries of $(K,g,J)$. \\
For the sake of clarity, we give an example of the group homomorphism $\psi$. Let $\varphi$ be a K\"ahler isometry of $K$. We may consider the map
$$
\psi:\Z^n \to {\mbox {Isom}}_{hol}(K),  m=(z_1,\dots,z_n) \mapsto \varphi^\abs{m}=\varphi^{z_1+\dots+z_n}.
$$
Then, $\psi(0)=\varphi^0=Id$ and $\psi(m+m')=\varphi^{z_1+w_1+\dots+z_n+w_n}=\varphi^\abs{m} \circ \varphi^\abs{m'}$. Hence $\psi$ is a group homomorphism. \\
We define the following $\Z^n$-action on $K \times G \times \R^n$ as $$m \cdot (p,q,t)=(\psi(m)p,\rho(m)q,t+m).$$ The action is clearly free, as
\[
(\psi(m)p,\rho(m)q,t+m)=(p,q,t) \implies t+m=t \iff m=0.
\]
We claim that the action is also properly discontinuous. First, observe that the action is manifestly smooth, as $\Z^n$ acts by isometries. Moreover, it is also easy to verify that the action is proper. Indeed, let $(p,q,t), (p',q',t') \in K \times G'$. Since the action (by translations) of $\Z^n$ on $\R^n$ is proper, there exist $I,I'$ open neighborhood of $t$ and $t'$ respectively such that the set $\Gamma:=\{ m \in \Z^n \ | \ (m+I)  \cap I' \neq \emptyset\}$ is finite. \\
Consider $U$ and $U'$ any open neighborhoods of $p$ and $p'$ and $V$ and $V'$ any open neighborhoods of $q$ and $q'$, respectively, and let $m \notin \Gamma$. Then
\[
\begin{split}
\big(m \cdot U \times V \times I \big) \cap U' \times V' \times I'&= \big( \psi(m) \, U \times \rho(m) V \times (m+I) \big) \cap U' \times V' \times I'\\
&=\big(\psi(m) \, U  \cap U' \big) \times  \big( \rho(m) V \cap V' \big) \times \big( (m+I) \cap I'\big) =\emptyset.
\end{split}
\]
It follows that 
\[
\Lambda:=\{ m \in \Z^n \ | \ \big(m \cdot U \times V \times I \big) \cap U' \times V' \times I'  \neq \emptyset\} \subset \Gamma,
\]
and therefore, since $\Gamma$ is finite, so is $\Lambda$.
\begin{lemma} \label{lemma:pluriclosed2}
Let $(G'=G \times \R^{n}, g'=b+g_E, J_L)$  be a Samelson space and let $\rho$ be a group homomorphism $\rho: \Z^n \to {\mbox {Isom}}(G)$ such that for each $m \in \Z^n$ the 
diffeomorphism $$m:G' \to G', \  (q,t) \mapsto (\rho(m)q, t+m)$$ is an holomorphic isometry of $(g', J_L)$. Let $(K,g,J)$ be a compact K\"ahler manifold and let $\psi$ be a group homomorphism $\psi: \Z^n \to {\mbox {Isom}}_{hol}(K)$, where ${\mbox{Isom}}_{hol}(K)$ is the group of holomorphic K\"ahler isometries of $(K,g,J)$.  
Then the quotient
\[
M_{\psi,\rho}= (K \times G \times \R^n)_{\psi,\rho} = \bigslant{K \times G \times \R^n} {\Z^n},  
\]
where $\Z^n$ acts freely and properly discontinuously  on $K \times G \times \R^n$ as \[m \cdot (p,q,t)=(\psi(m)p,\rho(m)q,t+m),\] admits a SKT structure $(I, h)$.
\end{lemma}
\begin{proof}
We prove the result by constructing a SKT structure $(\tilde{g}, \tilde{J})$ on $K\times G \times \R^n$ preserved by the action of $\Z^n$. 		\\
Consider the Hermitian structure $(\tilde{J}=J\times J_L, \tilde{g}=g+b+g_E)$  on $K\times G \times \R^n$ with corresponding fundamental form
$\tilde{\omega} =\omega+\omega_L.$
Using that $d\omega=0$, then $d^c_{\tilde{J}}\tilde{\omega} =d^c_{J_L}\omega_L$.\\
We compute $d^c_{J_L}\omega_L$. Let us fix $X,Y,Z$ left invariant vector fields on $G'$. Then, using the integrability of $J_L$ and the bi-invariance of the metric $g'$, we obtain
$
d^c_L\omega_L(X,Y,Z)=g'([X,Y],Z)
$
(for a more detailed computation of $d^c_L\omega_L$ see for instance \cite[Example 2.25]{MG}).\\
Since $d^c_L\omega_L$ is $Ad(G')$ invariant, $d^c_L\omega_L$ is a bi-invariant form of $G'$ and, hence, it is closed. Observe that $d^c_L\omega_L$ can be actually identified with a form on $G$, as the factor $\R^n$ is abelian. Indeed, let $X_i+Y_i$ be decomposable vector fields on $G'$, i.e., $X_i \in \mathfrak{X}(G)$ and $Y_i \in \mathfrak{X}(\R^n)$. Then,
\[
\begin{split}
d^c_L\omega_L(X_1+Y_1,X_2+Y_2,X_3+Y_3)=&g'([X_1+Y_1,X_2+Y_2],X_3+Y_3)\\
=&g'([X_1,X_2],X_3+Y_3) =b([X_1,X_2],X_3).
\end{split}
\]
To conclude the proof, we show that the Hermitian structure $(\tilde{g}=g+b+g_E, \tilde{J}=J\times J_L)$ is preserved by the action of  $\Z^n$. 	\\
Consider the diffeomorphism induced by $m \in \Z^n$, i.e. the map $m \cdot (p,q,t)= (\psi(m) p, \rho(m) q, t+ m)$. We want to prove that the action is actually holomorphic respect to $\tilde{J}$. By hypothesis, $\psi(m)$ is holomorphic with respect to $J$ and $\rho$ is such that $(\rho(m) q, t+m)$ is holomorphic with respect to $J_L$. Then the complex structure $\tilde{J}$ descends down on $M_{\psi,\rho}$ to a complex structure $I$.
 Moreover, since both the group homomorphisms $\psi$ and $\rho$ have image in the isometry group of $K$ and $G$ respectively, and $\Z^n$ acts on $\R^n$ by translations, the product metric $\tilde{g}=g+b+g_E$ is manifestly preserved by $\Z^n$, and hence descends to a metric $h$ which by construction is compatible with $I$.  Then, $(I, h)$ is a  SKT  structure on $M_{\psi,\rho}$.
\end{proof}
\begin{proposition} \label{prop:flatsam}
Let $M_{\psi,\rho}=(K \times G \times \R^n)_{\psi,\rho}$ be constructed as in Lemma \ref{lemma:pluriclosed2} and we assume that $(K,g)$ is a (non-flat) Ricci flat K\"ahler manifold. Then the SKT structure $(I, h)$ is  CYT  and the Bismut connection is non-flat.
\end{proposition}
\begin{proof}
To prove the Theorem, it suffices to prove that $(\tilde{g},\tilde{J})$ is a non-flat CYT and SKT structure on $K \times G \times \R^n$. Indeed, with the same argument used in the end of the proof of Lemma \ref{lemma:pluriclosed2}, $(\tilde{g},\tilde{J})$ descends on $M_{\psi,\rho}$ to the SKT structure $(I, h)$. \\
Since $(K,g,J)$ is K\"ahler, the Bismut connection of $(g,J)$  is the Levi-Civita  $\nabla^{LC}_1$.
Analogously, we denote by  $\nabla^B_{2}$ the Bismut connection of the Hermitian structure $(g',J_L)$ on $ G'$. $\nabla^B_{2}$ is defined by the relation
\[
    g'(( \nabla^B_{2})_{X}  Y,Z)= g'((\nabla^{LC}_{2})_{X}Y,Z)- \frac{1}{2} d^c_{J_{L}}\omega_L(X,Y,Z),
\]
for any $X,Y,Z \in \mathfrak{X}(G')$. In the following, we will denote by $H_2$ the torsion $3$-form $- d^c_{J_{L}}\omega_L$, and we have already shown that $H_2$ is a bi-invariant (and hence closed) form defined as $-g'([\cdot,\cdot],\cdot)$ on left-invariant vector fields.\\
Consider the Hermitian manifold $(K \times G', \tilde J,  \tilde {g})$.  We have already proved in Lemma \ref{lemma:pluriclosed2}, that the fundamental form $\tilde {\omega}$ satisfies 
\begin{equation} \label{eqn:torsion}
d^c_{\tilde {J}} \tilde \omega =d^c_{\tilde {J}}( \omega+\omega_L)=d^c_{J_L}\omega_L=-H_2.
\end{equation}
Consider  the  unique metric connection  $\tilde {\nabla}$ on $(K \times G',  \tilde J, \tilde {g})$ with skew-symmetric torsion $H_2$ determined by
\[
    \tilde g (\tilde {\nabla}_{X}Y,Z)=  \tilde g  (\tilde {\nabla}^{LC}_{X}Y,Z)- \frac{1}{2} d^c_{J_{L}}\omega_L(X,Y,Z),  \ \forall X,Y,Z \in \mathfrak{X}(K \times G'), 
\]
where $\tilde \nabla^{LC}$ is the Levi-Civita connection of the product metric $ \tilde g $.  It turns out that $\tilde \nabla$ is the Bismut connection of  $(K \times G', \tilde J, \tilde g)$, by \eqref{eqn:torsion}. From now on we set $\tilde \nabla=\tilde \nabla^B$ and we will denote by $\tilde{H}$ the torsion $3$-form $-d^c_{\tilde {J}} \tilde \omega$ of $\tilde\nabla^B$. \\
Let now compute  the Bismut curvature tensor $\tilde {R}^B$. Using the same computation of the proof of Proposition \ref{prop:flathopf}, we get 
\begin{align*}
 \tilde R^B (X_1+X_2,Y_1+Y_2, Z_1+Z_2,W_1+W_2) = &  R^{LC}_{1}(X_1,Y_1,Z_1,W_1)+R^B_{2}(X_2,Y_2,Z_2,W_2)\\[3pt]
  =& R^{LC}_{1}(X_1,Y_1,Z_1,W_1),
 \end{align*}
where the last equality follows  from the fact that the Bismut curvature tensor $R^B_{2}$ of $(g', J_L)$ vanishes, as $g'$ is bi-invariant and $J_L$ is left-invariant (for a reference see for instance \cite[Proposition 8.39]{GFS}).\\
We are now ready to compute the Ricci form $\tilde \rho^B$ of the Bismut connection $\tilde \nabla^B$. As already done before, without loss of generality, we may compute $\tilde \rho^B$ on decomposable vector fields. By definition 
\[
\begin{split}
\tilde \rho^B (X_1+X_2,Y_1+Y_2)=& \frac{1}{2} \sum_{i=1}^{2k+2r} \tilde R^B (X_1+X_2,Y_1+Y_2,J e_i ,e_i)\\
=&  \frac{1}{2} \sum_{i=1}^{2k} R^B_1 (X_1,Y_1,J e_i ,e_i) 
=  \rho^{LC}_{1} (X_1,Y_1)
= 0,
\end{split}            
\]
where $ \rho^{LC}_{1}$ is the Ricci curvature form of $(K,g,J)$, $\{e_1,\dots,e_{2k}\}$ is a local orthonormal basis of $TK$ and $\{e_{2k+1},\dots,e_{2k+2r} \}$ is a local orthonormal basis of $TG'$. 	\\
The pair $( \tilde J, \tilde g)$ descends on $M_{\psi,\rho}$ to $(I, h)$, where $(I, h)$ is the SKT  structure constructed in Lemma \ref{lemma:pluriclosed2}, by previous remarks. We then have that  $(I, h)$ is CYT structure with a non-flat Bismut connection.
\end{proof}
\section{Generalized K\"ahler examples}
\begin{definition} [\cite{MG2}]
A bi-Hermitian manifold $(M, I_{\pm}, g)$ is said to be \emph{generalized K\"ahler} if $g$ is $I_\pm$-compatible, $d^c_+ \omega_+=-d^c_-\omega_-$ and $dd^c_+ \omega_+=0$, where $d^c_\pm=-I_\pm d\omega_\pm$ and $\omega_\pm=g I_{\pm}$.
\end{definition}
 As a trivial example, if $(g, J)$ is a K\"ahler structure on $M$, $J_+ = J$ and $J_- = \pm J$ is a solution of the above equations. Hence, the case of major interested is when the generalized K\"ahler structure does not come from a K\"ahler one. We refer to such a generalized K\"ahler structure as \emph{non-trivial}. We recall  the following
\begin{definition}
A generalized K\"ahler structure $(I_{\pm}, g)$ is said to be \emph{twisted} if $[d^c_+ \omega_+] \neq 0 \in H^3(M)$ and \emph{untwisted} otherwise.
\end{definition}
Given a generalized K\"ahler manifold $(M,g, J_\pm)$, the closed $3$-form $d^c_+\omega_+=-d^c_-\omega_-$ is also called the \emph{torsion} of the generalized K\"ahler structure. \\ 
Consider the SKT  structure $(h,I_-)$, with $I_- = I$,  constructed in Lemma \ref{lemma:pluriclosed}. In the next Theorem we prove that on $(M_f,h)$ is always possible to find another complex structure $I_+$ compatible with $h$ and satisfying $d^c_+ \omega_+=-d^c_-\omega_-$. Therefore, the triple $(h,I_\pm)$ defines a generalized K\"ahler structure on $M_f$. 

\begin{theorem} \label{theorem:kahlermappingtorus}
Let $(K,J,g)$ be a compact K\"ahler manifold of complex dimension $k$ and let $\psi$ be a K\"ahler isometry, i.e., $\psi:K \to K$ is an holomorphic diffeomorphism satisfying $\psi^*(g)=g$. Then, the mapping torus 
$
M_f=(K \times \mathbb{S}^3)_f,
$
with $f=(\psi,Id_{\Sph}^3)$, admits a split twisted generalized K\"ahler structure $(I_{\pm}, h, \Omega_{\pm})$.
Moreover, $I_+$  and $I_-$  induce opposite orientations on $M_f$.
\end{theorem}
\begin{proof}
Let $(I_- := I, h)$ the  SKT structure constructed in Lemma  \ref{lemma:pluriclosed}.
 We consider on $K \times \C^2 \setminus \{(0,0)\} \cong K \times \Sph^3 \times \R$ another product complex structure $J \times J_+$, where $J_+$ is the complex structure on $\C^2 \setminus \{(0,0)\}$ obtained by changing the orientation of the $z_2$ plane. More precisely, if $(z_1,z_2)$ and  $(\zeta^1,\zeta^2)$ are the holomorphic coordinates associated to $J_{\pm}$ respectively, then $\zeta_1=z_1$ and $\zeta_2=\overline{z_2}$.\\
The automorphisms $\phi_n$ associated to the $\Z$-action described in Lemma \ref{lemma:pluriclosed} are holomorphic also with respect to $J \times J_+$, implying that $J \times J_+$ descends to a complex structure $I_+$ on the quotient $ M_f\cong\bigslant{K\times \C^2\setminus\{(0,0)\} }{ \Z} $, which satisfy $[I_+,I_-]=0$.\\
We denote by $\Omega_+=\omega+\omega_+$ the fundamental form of the Hermitian structure $(I_+, h)$, where $h$ is the Riemannian metric explicitly given in \eqref{eqn:gkmetric}  and $\omega_+$ can be written in the coordinates $(\zeta_1,\zeta_2)$ and $(z_1,z_2)$, respectively, as
\[
\omega_+=\frac{i}{R^2} (d\zeta_1 \wedge d \overline{\zeta}_1 +  d\zeta_2 \wedge d \overline{\zeta}_2)=\frac{i}{R^2} (dz_1 \wedge d \overline{z}_1 -  dz_2 \wedge d \overline{z}_2).
\]
We compute $d^c_+ \Omega_+$. 
\begin{align*}
d^c_+ \Omega_+= &- \frac{1}{R^4}\left[(\overline\zeta_2d\zeta_2-\zeta_2 d\overline{\zeta}_2)\wedge d\zeta_1 \wedge d\overline{ \zeta}_1+ (\overline{\zeta}_1d\zeta_1-\zeta_1 d\overline{\zeta}_1)\wedge d\zeta_2 \wedge d\overline{ \zeta}_2\right]  \\ 
=& \quad \frac{1}{R^4}\left[(\overline{z}_2dz_2-z_2 d\overline{z}_2)\wedge dz_1 \wedge d\overline{ z}_1+ (\overline{z}_1dz_1-z_1 d\overline{z}_1)\wedge dz_2 \wedge d\overline{ z}_2\right]\\
=& -d^c_- \Omega_-.
\end{align*}
It follows that $(h,I_\pm)$ is a split generalized K\"ahler structure on $M_f \cong \bigslant{K\times \C^2\setminus\{(0,0)\} }{ \Z}$. \\
We claim that $(h,I_\pm)$  is a twisted  generalized K\"ahler structure. \\
Let $H=d^c_+ \Omega_+$. If one considers the radial projection 
\[
\pi: \bigslant{K\times \C^2\setminus\{(0,0)\} }{ \Z} \to \Sph^3, \quad 
    [(p,\underline{x})] \mapsto \frac{\underline{x}}{\norm{\underline{x}}},
\]
then it is straightforward to observe that $H=2 \pi^* vol_{\Sph^3}$. \\
By contradiction, let us assume that $H$ is exact. Fixed any $p \in K$ and $t \in (0,1)$, we define $\iota_{p,t}:\Sph^3 \to  \bigslant{K\times \C^2\setminus\{(0,0)\} }{ \Z}, \ \ q \mapsto [(p, 2^t \cdot q)]$. Then, by Stokes Theorem,
\[
0= \int_{\Sph^3} \iota_{p,t}^*H=2 \int_{\Sph^3} \iota_{p,t}^* ( \pi^* vol_{\Sph^3} )=2  \int_{\Sph^3} (\pi \circ \iota_{p,t})^* vol_{\Sph^3}=2 \int_{\Sph^3} vol_{\Sph^3}\neq 0.
\]
Clearly, this leads to a contradiction. \\
Let us consider the volume forms associated to the pairs $(h, I_\pm)$, which are respectively
\[
\begin{split}
\frac{1}{k+2!}\Omega_-^{k+2}=&-\frac{1}{ k! R^4} \omega^k \wedge dz_1 \wedge d\overline{z}_1 \wedge dz_2 \wedge d\overline{z}_2,		\\
\frac{1}{k+2!}\Omega_+^{k+2}=&-\frac{1}{ k! R^4} \omega^k \wedge d\zeta_1 \wedge d\overline{\zeta}_1 \wedge d\zeta_2 \wedge d\overline{\zeta}_2.\\
\end{split}
 \]
Since $d\zeta_1= dz_1$ and $d \zeta _2= d\overline{z}_2$, $\frac{1}{k+2!}\Omega_-^{k+2}=-\frac{1}{k+2!}\Omega_+^{k+2}$. The last statement follows. 
\end{proof}
\begin{remark}
By Proposition \ref{proposition:prop}, the generalized K\"ahler structure constructed in Theorem \ref{theorem:kahlermappingtorus} is not trivial. 
\end{remark}
We exhibit an explicit Example fitting in the hypothesis of Theorem \ref{theorem:kahlermappingtorus}.
\begin{example}
Let $K$ be the flat torus $\mathbb{T}^4$ endowed with the standard K\"ahler structure $(g,J,\omega)$, defined as follows
\[
J\bigg(\frac{\partial}{\partial x_1}\bigg)=\frac{\partial}{\partial x_2}, \quad J\bigg(\frac{\partial}{\partial x_3}\bigg)=-\frac{\partial}{\partial x_4}, \quad 
g=dx_1^2+dx_2^2+dx_3^2+dx_4^2,  \quad
\omega=dx_1\wedge dx_2-dx_3\wedge dx_4,
\]
where $(x_1,x_2,x_3,x_4)$ are the standard coordinates on $\mathbb{T}^4 \cong \mathbb{Z}^4 \backslash \mathbb{R}^4$. \\
Let $\psi$ be the $\mathbb{R}^4$-rotation $(x_1,x_2,x_3,x_4)\mapsto (x_2,-x_1,x_4,-x_3)$. Since $\psi$ is represented by an integer matrix, $\psi$ descend to a diffeomorphism of the flat torus $\mathbb{T}^4$. \\
It is immediate to observe that $\psi$ is holomorphic with respect to $J$ and preserves the K\"ahler structure $(g,J)$. Indeed, $[\psi_*, J] =0$ and
$\psi^* g= \psi^*\big(\sum_{i=1}^4 (dx_i)^2 \big)=g.$

Then the mapping torus
\[
   M_f= \frac{\mathbb{T}^4 \times \Sph^3 \times [0,1]}{(p,q,0) \sim (\psi(p),q,1) },
\]
is a generalized K\"ahler manifold, by Theorem \ref{theorem:kahlermappingtorus}. \\
We compute the cohomology of the example. Since $M_f\cong M_\psi \times \Sph^3$ one may compute the cohomology of $M_\psi$, and then apply the K\"unneth formula. \\
As already mentioned in the proof of the Proposition \ref{proposition:prop}, the cohomology of $M_\psi$ is completely determined by the data $N^r_{\mathbb{T}^4}=\ker(\psi_r^*-Id)$ and $C^r_{\mathbb{T}^4}=\operatorname{coker}(\psi_r^*-Id)$, which are easily computable:
\[
\begin{split}
     &K^0_{\mathbb{T}^4}= C^0_{\mathbb{T}^4}= \langle 1 \rangle, \ \ \ N^1_{\mathbb{T}^4}= C^1_{\mathbb{T}^4}= 0,  \\
     &N^2_{\mathbb{T}^4}= C^2_{\mathbb{T}^4}= \langle \ [dx_{12}], [dx_{13}+dx_{24}], [dx_{14}-dx_{23}], [dx_{34}] \ \rangle,\\
     &N^3_{\mathbb{T}^4}= C^3_{\mathbb{T}^4}= 0, \ \ \  N^4_{\mathbb{T}^4}= C^4_{\mathbb{T}^4}= \langle \  [dx_{1234}] \ \rangle. \\
\end{split}
\]
The computation of $H^\bullet(M_\psi)$ is now trivial,   
$$
\begin{array}{l}
    H^1 (M_\psi) =  \langle [dt]   \rangle, \quad
    H^2 (M_\psi) =  \langle  [dx_{12}], [dx_{13}+dx_{24}], [dx_{14}-dx_{23}], [dx_{34}]  \rangle,\\[3pt]
    H^3 (M_\psi) =  \langle [dt\wedge dx_{12}], [dt \wedge (dx_{13}+dx_{24})], [dt\wedge(dx_{14}-dx_{23})], [dt\wedge dx_{34}] \rangle,\\[3pt]
    H^4 (M_\psi) =  \langle [dx_{1234}]    \rangle,  \quad
    H^5(M_\psi)  =  \langle [dt \wedge d_{1234}]    \rangle.
\end{array}
$$
\end{example}
Consider the $2r$-dimensional Lie group $G'=G \times \R^{n}$ 
endowed with the bi-invariant metric $g'=b+g_{E}$, where $G$ is a compact, conneted and simply connected semisimple Lie group endowed with the bi-invariant metric $b$ and $g_E$ is the Euclidean metric on $\R^n$. We assume that $\dim(G) \ge 3$. We denote by $\mathfrak{g}^L(G')$ and $\mathfrak{g}^R(G')$ the left and right Lie algebras of $G$, respectively. Then
\[
\mathfrak{g}^L(G') = \mathfrak{g}^L \oplus \R^n = \mathfrak{g}^L(G''), \ \mathfrak{g}^R(G') = \mathfrak{g}^R \oplus \R^n = \mathfrak{g}^R(G''),
\]
where $G''=G \times \mathbb{T}^n$. Therefore $G''$ inherits form $G'$ the bi-invariant metric $g'$. Since $G''$ is compact, it admits both a left and a right invariant complex structure, namely $J_L$ and $J_R$,  which are compatible with the bi-invariant metric $g'$. As $G'$ and $G''$ shares the same Lie algebras, then $(g',J_L,J_R)$ is also a bi-Hermitian structure on $G'$, where $J_L$ and $J_R$ are left and right invariant complex structures, respectively. In particular, $(G', J_L, g')$ is a Samelson space. \\

\begin{theorem} \label{theorem:samspaces}
Let $(G'=G \times \R^{n}, g', J_L, J_R)$  be as above. Let $\rho$ be a group homomorphism $\rho: \Z^n \to {\mbox {Isom}} (G)$ such that for each $m \in \Z^n$ the 
diffeomorphism \[m \cdot (q,t)= (\rho(m)q, t+m)\] is a holomorphic isometry of $g'$ with respect to $J_L$ and $J_R$. Let $(K,g,J)$ be a compact K\"ahler manifold and let $\psi$ be a group homomorphism $\psi: \Z^n \to I{\mbox {som}}_{hol}(K)$, where ${\mbox {Isom}}_{hol}(K)$ is the group of holomorphic K\"ahler isometries of $(K,g,J)$.  
Then the quotient
\[
M_{\psi,\rho}= (K \times G \times \R^n)_{\psi,\rho} = \bigslant{K \times G \times \R^n} {\Z^n},  
\]
where $\Z^n$ acts freely and properly discontinuously  on $K \times G \times \R^n$ as $$m \cdot (p,q,t)=(\psi(m)p,\rho(m)q,t+m),$$ admits a generalized K\"ahler structure.
\end{theorem}
\begin{proof}
As already seen in the proof of Lemma \ref{lemma:pluriclosed2}, $(\tilde{g}=g+g', \tilde{J}_-=\tilde{J}=J\times J_L)$ is a SKT structure on $K \times G \times \R^n$ preserved by the $\Z^n$-action $m \cdot (p,q,t)= (\psi(m)p, \rho(m)q, t+m)$, which induces the SKT structure $(h, I_-=I)$ on $M_{\psi,\rho}$ (for the notation, see Lemma \ref{lemma:pluriclosed2}). Consider the following product complex structure $\tilde{J}_+=J\times J_R$ which is compatible with $\tilde g$ by construction and it is such that the fundamental form of $(\tilde g, \tilde J_+)$ is
$$
 \tilde{\omega}_+=\omega+\omega_R.
$$
Using that $d\omega=0$, then $d^c_{\tilde{J}_+}\tilde{\omega}_+=d^c_{J_R}\omega_R$.  Recalling that $d^c_{J_L}\omega_L=g'([\cdot,\cdot],\cdot)$ on left invariant vector fields, we obtain $d^c_{\tilde{J}_-}\tilde{\omega}_-=d^c_{J_L}\omega_L=-d^c_{J_R}\omega_R$, as the right Lie algebra is anti-isomorphic to the left one. 
To conclude the proof, it suffices to show that $\tilde J_+$ is preserved by the $\Z^n$-action. Moreover, by hypothesis, $\psi(m)$ is holomorphic with respect to $J$ and $\rho$ is such that $(\rho(m) q, t+m)$ is holomorphic with respect to both $J_L$ and $J_R$. Then the complex structure $\tilde{J}_+$ descends on $M_{\psi,\rho}$ to a complex structure $I_+$. Therefore, $M_{\psi,\rho}$ inherits from $K \times G \times \R^n$ the generalized K\"ahler structure. 
\end{proof}
\begin{remark}
If $\rho$ is the trivial homomorphism, i.e. $\rho (m)=Id$, for each $m\in \Z^n$, then the induced diffeomorphism $(q,t)\mapsto (q, t+m)$ is necessarily holomorphic with respect to $J_L$ and $J_R$. 
\end{remark}
\begin{remark}
The complex structures of generalized K\"ahler metric $(h,I_\pm)$ induce the same orientation. Indeed, $I_\pm$ are induced by $\tilde{J}_-=J \times J_L$ and $\tilde{J}_+=J \times J_R$ and $\tilde J_\pm$ have the same orientation, as $J_L$ and $J_R$ are isomorphic as complex manifolds via the inversion of the group. In the case $G=SU(2)$ and $n=1$, then Theorem \ref{theorem:samspaces} gives a different generalized K\"ahler structure on the mapping torus $M_f$ with respect to Theorem \ref{theorem:kahlermappingtorus}. 
\end{remark}

\begin{corollary}
Let $(M_{\psi,\rho},h,I_\pm)$ be the generalized K\"ahler manifold constructed as in Theorem \ref{theorem:samspaces}. Then the generalized K\"ahler structure $(h,I_\pm)$ is twisted, $I_+$ and $I_-$ both fail to satisfy the $dd^c_\pm$-Lemma and in particular they do not admit any compatible K\"ahler metric. 
\end{corollary}
\begin{proof}
By contradiction, assume that the torsion 3-form $H$ of the generalized K\"ahler structure $(h,I_\pm)$ on $M_{\psi,\rho}$ is exact. Then, defined by $\pi$ the covering map $\pi:K \times G \times \R^n \to M_{\psi,\rho}= (K \times G \times \R^n)_{\psi,\rho}$, we have that $\pi^* H$ is also exact. By construction $\pi^* H=\tilde{H}$, where $\tilde H=d^c_{\tilde J_+} \tilde \omega_+=d^c_{J_R}\omega_R$, and so $\tilde H$ can be identified with a $3$-form on $G$ by previous remarks. By the exactness of $H$, $[\tilde{H}] =0 \in H^3(G')\cong H^3(G)$, by K\"unneth formula. Moreover, since $G$ is compact, $H^3(G) \cong \Omega_I(G)$, where  $ \Omega_I(G)$ is the complex of bi-invariant forms. It then follows that $\tilde{H}$ must be the zero form on $G$ and hence $[X,Y]=0$ for any pair of left-invariant vector fields in $G$. This would implies that the Lie algebra of $G$ is abelian, but this provides a contradiction, as $G$ is semisimple. The result follows by applying \cite[Corollary 2.19]{MG}.
\end{proof}

\section{Dolbeault cohomology}
We give now a description of $M_f$ as the total space of a holomorphic fibre bundle $p:M_f \to \Sph^3 \times \Sph^1$ with fibre $K$. To such fibre bundle it is always possible to associate the Borel spectral sequence, which relates the Dolbeault cohomology of the total space $M_f$ with that of the base space $\Sph^3 \times \Sph^1$ and that of the fibre $K$. \\
We first recall the following Theorem of Borel contained in \cite[Appendix II]{HB}. 
\begin{theorem} \label{theorem:Borel}
Let $p:T \to B$ be a holomorphic fibre bundle, with compact connected fibre $F$ and $T$ and $B$ connected. Assume that $F$ is K\"ahler. Then there exists a spectral sequence $(E_r,d_r)$, with $d_r$ being the restriction of the debar operator $\overline{\partial}$ of $T$ to $E_r$, satisfying the following properties:
\begin{enumerate}
    \item $E_r$ is $4$-graded by the fibre degree, the base degree and the type. Let $^{p,q}E_r^{u,v}$ be the subspace of elements of $E_r$ of type $(p,q)$, fibre degree $u$ and base degree $v$. We have that $^{p,q}E_r^{u,v}=0$ if $p+q \neq u+v$ or if one of $p,q,u,v$ is negative. Moreover, $d_r$ maps $^{p,q}E_r^{u,v}$ into $^{p,q+1}E_r^{u+r,v-r+1}$. \\
    \item If $p+q=u+v$
    \[
    ^{p,q}E_2^{u,v}=\sum_{k} H^{k,u-k}_{\overline{\partial}}(B) \otimes H^{p-k,q-u+k}_{\overline{\partial}}(F).
    \]
    \item The Borel spectral sequence converges to $H_{\overline{\partial}}(T)$.
\end{enumerate}
\end{theorem} 
Let us denote by $\pi$ the well-defined projection
$$
\pi:M_f \to \Sph^3 \times \Sph^1, 
 [(p,q,t)] \mapsto [(p,t)].
$$
We always assume that $M_f$ is endowed with the standard complex structure $I$, which we recall to be induced by $\tilde{J}=J \times J_-$, and $\Sph^3 \times \Sph^1 $ is endowed with the standard complex structure induced by $J_-$, i.e.,  the standard complex structure of its universal cover. With respect to such complex structures, the bundle map $\pi$ is holomorphic. Now we exhibit a local trivialization around each point $[(q,t)]$ of $\Sph^3 \times \Sph^1$. \\
First we consider points of the kind $[(q,t)]$ with $t \neq 0,1$. Let $U=V \times (t-\varepsilon,t+\varepsilon)$, where $V$ is an open neighborhood of $q$ and $\varepsilon$ is such that $(t-\varepsilon,t+\varepsilon)$ does not contain $0$ and $1$.  Then $\pi^{-1} (U)=K \times V \times (t-\varepsilon,t+\varepsilon)$. We define the local trivialization
$$
\phi_U: \pi^{-1} (U) \to U \times K, 
[(p,q,t)] \mapsto ([(q,t)],p),
$$
which clearly is a biholomorphism. \\
Now we consider points of the kind $[(q,0)]$. An open neighborhood of $[(q,0)]$ is given by $U=\pi'(V \times [0,\varepsilon) \sqcup V \times (1-\varepsilon,1])$ where $V$ is an open neighborhood of $q$ in $\Sph^3$, $\abs{\varepsilon} < \frac{1}{2}$ and $\pi': \Sph^3 \times \R \to  \Sph^3 \times \Sph^1$ is the standard quotient map induced by the $\Z$-action on $\R$ by traslations. \\
Therefore, $\pi^{-1}U=\pi''(K \times V \times [0,\varepsilon) \sqcup K \times V \times (1-\varepsilon,1])$, where $\pi''$ is the mapping torus map $K\times \Sph^3 \times \Sph^1 \to M_f$.  The local trivialization is defined on the representatives as
\begin{align*}
\phi_U:\pi^{-1}U &\to U \times K \\
[(p,q,t)] &\mapsto \begin{cases}  
                              ([(q,t)],p) \ \text{if} \ t \in [0,\varepsilon) \\
                              ([(q,t)],\psi^{-1} (p)) \ \text{if} \ t \in (1-\varepsilon,1 ] .
                              \end{cases}
\end{align*}    
Although the definition of $\phi_U$ depends on the representative chosen, it is immediate to prove that it is well posed. Indeed
\[
\phi_U[(p,q,0)]= ([(q,0)],p)=([(q,1)],p)=\phi_U[(\psi(p),q,1)].
\]   
Moreover, since $\psi$ is holomorphic respect to $J$, the holomorphy of $\phi_U$ follows. We exhibit now the inverse for $\phi_U$. We define
\begin{align*}
\phi_U^{-1}: U \times K &\to  \pi^{-1}U  \\
([(q,t)],p) &\mapsto \begin{cases}  
                              [(p,q,t)] \ \text{if} \ t \in [0,\varepsilon) \\
                              [(\psi (p),q,t)] \ \text{if} \ t \in (1-\varepsilon,1 ] .
                              \end{cases}
\end{align*}  
Again, $\phi_U^{-1}$ is well defined
\[
\phi_U^{-1}([(q,0)],p) =[(p,q,0)]=[(\psi(p),q,1)]=\phi_U^{-1}([(q,1)],p),
\]
and it is holomorphic since so is $\psi^{-1}$. It remains to prove that $\phi_U \circ \phi_U^{-1}=Id$ and $\phi_U^{-1} \circ \phi_U=Id$.  Although it is a straightforward check, we report it here for the sake of completeness. 
\[
[(p,q,t)] \xrightarrow{\phi_U} 
                               \begin{cases}  
                              ([(q,t)],p) \ \text{if} \ t \in [0,\varepsilon) \\
                              ([(q,t)],\psi^{-1} (p)) \ \text{if} \ t \in (1-\varepsilon,1 ] .
                              \end{cases}
 \xrightarrow{\phi_U^{-1}}
 [(p,q,t)] 
\]
and 
 \[
([(q,t)],p) \xrightarrow{\phi_U^{-1}}
 \begin{cases}  
[(p,q,t)] \ \text{if} \ t \in [0,\varepsilon) \\
[(\psi (p),q,t)] \ \text{if} \ t \in (1-\varepsilon,1 ] \\
 \end{cases}
\xrightarrow{\phi_U}
([(q,t)],p) .
\]
The following result easily follows
\begin{theorem} \label{theorem:holfib}
If $M_f$ is the mapping torus constructed as in the Theorems \ref{theorem:kahlermappingtorus}, then $(M_f,I)$ is the total space of the holomorphic fibre bundle
\[
   (K,J) \to (M_f,I) \to (\Sph^3 \times \Sph^1,J_-).
\]
Then we have an associated Borel spectral sequence $(E_r,d_r)$ satisfying the following properties:
\begin{enumerate}
\item If $p+q=u+v$
    \[
    ^{p,q}E_2^{u,v  }=\sum_{k} H^{k,u-k}_{\overline{\partial}}(\Sph^3 \times \Sph^1) \otimes H^{p-k,q-u+k}_{\overline{\partial}}(K).
    \] 
\item The Borel spectral sequences converges to $H^{\bullet, \bullet}_{\overline{\partial}}(M_f)$. 
\end{enumerate} 
\end{theorem}
\begin{corollary}
The mapping tori $(M_f, I)$ do not admit any balanced metrics.  
\end{corollary}
\begin{proof}
The projection map $p:(M_f, I) \to (\Sph^3 \times \Sph^1, J_-)$ is an holomorphic submersion, which is proper, as $M_f$ is compact. Since $(\Sph^3 \times \Sph^1, J_-)$ does not admit any balanced metric, the result follows by applying \cite[Proposition 1.9]{MI}.

\end{proof}

\smallskip
\textbf{Acknowledgements}. 
The authors would like to thank Jeffrey Streets and Mario Garcia-Fernandez for their interest in this paper and Leander Stecker for kindly pointing out reference \cite{AFF}. The first author would also like to thank Tommaso Sferruzza for useful discussions and suggestions.  
Beatrice Brienza and Anna Fino are partially supported by Project PRIN 2022 \lq \lq Geometry and Holomorphic Dynamics” and by GNSAGA (Indam). Anna Fino   is also supported  by a grant from the Simons Foundation (\#944448). She would like also to thank  Simon Foundation for the  support  to the  participation at  the  MATRIX Research Program "Spectrum and Symmetry for Group Actions in Differential Geometry II,"
 24 July - 4 August 2023. 
Gueo Grantcharov is partially supported by a grant from the Simons Foundation (\#853269).

\end{document}